\documentclass[a4paper]{amsart}
\usepackage[a4paper, margin=2cm]{geometry}
\usepackage[british]{babel}

\usepackage{amsmath}
\usepackage{amssymb}
\usepackage{amsthm}
\usepackage{calc}
\usepackage{enumitem}
\usepackage{graphicx}
\usepackage{url}
\usepackage{xcolor}
\usepackage{etoolbox}
\usepackage{tabularx}
\usepackage{mathtools}
\usepackage{hyperref}

\hypersetup{
  colorlinks   = true, %Colours links instead of ugly boxes
  urlcolor     = {blue!50!black}, %Colour for external hyperlinks
  linkcolor    = {blue!50!black}, %Colour of internal links
  citecolor   = {red!50!black} %Colour of citations
}

\usepackage{cleveref}

\numberwithin{equation}{section}

\theoremstyle{plain}
\newtheorem{thm}{Theorem}[section]

\newtheorem{prop}[thm]{Proposition}
\newtheorem{lem}[thm]{Lemma}
\newtheorem{conj}[thm]{Conjecture}
\newtheorem{fact}[thm]{Fact}
\newtheorem*{lem*}{Lemma}
\newtheorem*{cor*}{Corollary}
\newtheorem*{thm*}{Theorem}
\newtheorem*{fact*}{Fact}

\newtheorem{manualthminner}{Theorem}
\newenvironment{manualthm}[1]{%
  \IfBlankTF{#1}
    {\renewcommand{\themanualthminner}{\unskip}}
    {\renewcommand\themanualthminner{#1}}%
  \manualthminner
}{\endmanualthminner}

\theoremstyle{definition}
\newtheorem{defn}[thm]{Definition}
\newtheorem{eg}[thm]{Example}
\newtheorem{notn}[thm]{Notation}
\newtheorem{qn}[thm]{Question}

\theoremstyle{remark}
\newtheorem{rem}[thm]{Remark}
\newtheorem*{term*}{Terminology}

\DeclareMathOperator{\Age}{Age}

\DeclareMathOperator{\Aut}{Aut}
\DeclareMathOperator{\Can}{Can}
\DeclareMathOperator{\dom}{dom}

\DeclareMathOperator{\id}{id}
\DeclareMathOperator{\im}{im}

\DeclareMathOperator{\qftp}{qftp}

\newcommand{\N}{\mathbb{N}}
\newcommand{\Q}{\mathbb{Q}}
\newcommand{\R}{\mathbb{R}}

\newcommand{\ic}{\mathbin{\bot}}
\newcommand{\mb}[1]{\mathbb{#1}}
\newcommand{\mc}[1]{\mathcal{#1}}

\newcommand{\all}{\forall\,}
\newcommand{\ex}{\exists\,}
\newcommand{\sub}{\subseteq}

\newcommand{\fin}{\subseteq_{\text{fin\!}}}

\newcommand{\ra}{\rightarrow}

\newcommand{\la}{\leftarrow}

\newcommand{\Mod}[1]{(\mathrm{mod}\ #1)}

\newcommand{\eps}{\varepsilon}

\newcommand{\Fr}{Fra\"{i}ss\'{e} }

\newcommand{\Jarik}{Ne\v{s}et\v{r}il }
\newcommand{\NR}{Ne\v{s}et\v{r}il-R\"{o}dl }

\allowdisplaybreaks

\subjclass[2020]{05C55, 03C15, 05D05, 05D10, 05C63}

\keywords{sunflower, delta-system, vertex partition, Ramsey theory, ultrahomogeneous}

\makeatletter
\def\author@andify{%
  \nxandlist {\unskip ,\penalty-1 \space\ignorespaces}%
    {\unskip {} \@@and~}%
    {\unskip \penalty-2 \space \@@and~}%
}
\makeatother

\author{Rob Sullivan}
\address{\parbox{\linewidth}{Rob Sullivan\\
Institute of Computer Science, Czech Academy of Sciences\\
Pod Vodárenskou věží 271/2,\\
182 00 Prague\\
Czech Republic
}}
\email{robertsullivan1990+maths@gmail.com}

\author{Jeroen Winkel}
\address{Jeroen Winkel}
\email{winkeljeroen+maths@gmail.com}

\thanks{The first author is funded by Project 24-12591M of the Czech Science Foundation (GA\v{C}R)}

\title{Structured sunflowers and canonical Ramsey properties}

\begin{document}

\begin{abstract}
    A first-order structure $M$ is said to have the \emph{infinite sunflower property} if, for each $k \in \N_+$ and each structure $M' \cong M$ whose elements are $k$-sets, there is $S \sub M'$, $S \cong M$, such that $S$ is a \emph{sunflower}: a collection of sets such that each pair of elements has the same intersection. A class $\mc{K}$ of finite structures is said to have the \emph{finite sunflower property} if for all $k \in \N_+$ and $B \in \mc{K}$, there is $C \in \mc{K}$ such that any structure $C' \cong C$ whose elements consist of $k$-sets contains a copy of $B$ which is a sunflower. These two notions were introduced by Ackerman, Karker and Mirabi in a recent paper, and give a structural generalisation of the well-known Erd\H{o}s-Rado sunflower lemma for sets. We show two results for countable ultrahomogeneous relational structures with strong amalgamation: first, the infinite sunflower property is equivalent to the canonical infinite point-Ramsey property; second, a certain strengthening of the canonical finite point-Ramsey property implies the finite sunflower property. (Here, ``canonical" refers to statements analogous to the Erd\H{o}s-Rado canonical Ramsey theorem, involving colourings with infinitely many colours.) We also show that all free amalgamation classes with a single vertex isomorphism-type have the finite sunflower property, as do many classes of finite metric spaces, and we give a variety of further examples and observations.
\end{abstract}

\maketitle

\section{Introduction} \label{s: Introduction}

A \emph{sunflower}, or \emph{$\Delta$-system}, is a collection $S$ of sets such that any pair of sets in $S$ have the same intersection. We call the elements of $S$ \emph{petals}. In \cite{ER60}, Erd\H{o}s and Rado proved the following result, often known as the \emph{sunflower lemma}, with applications in combinatorics, set theory and computer science (see \cite{ALWZ20}):
\begin{fact*}[{\cite[Thm.\ 1]{ER60}}] \hfill
    \begin{enumerate}[label=(\roman*)]
    \item For all $k \in \N_+$, any countably infinite collection of sets, each of size $k$, contains a sunflower with infinitely many petals.
    \item For all $b, k \in \N_+$, there exists $c \in \N_+$ such that any collection of $c$-many sets, each of size $k$, contains a sunflower with $b$ petals.
\end{enumerate}     
\end{fact*}

(Erd\H{o}s and Rado in fact proved much more: they considered infinite cardinals in general, and gave an explicit bound for $c$ in the finite case.)

In this paper, we investigate a generalisation of the above to the case of first-order structures, extending results of \cite{AKM25}. In \cite[Def.\ 6]{AKM25}, Ackerman, Karker and Mirabi introduce the notion of \emph{structured sunflowers}:

\begin{defn}
    Let $k \in \N_+$. We refer to sets of size $k$ as \emph{$k$-sets}. We say that a first-order structure $M$ is a \emph{structure on $k$-sets} if each element of the domain of $M$ is a $k$-set. Let $S$ be a structure on $k$-sets. We say that $S$ is a \emph{structured sunflower} if there is a set $c$ with $v \cap v' = c$ for all distinct $v, v' \in S$. We call $c$ the \emph{core} of the sunflower.
\end{defn}

Our central aim is to determine under what conditions structured sunflowers are guaranteed to exist.

\begin{defn}[{\cite[Def.\ 6]{AKM25}}]
    Let $k \in \N_+$. Let $M$ be an infinite first-order structure. We say that $M$ has the \emph{infinite $k$-sunflower property} if each structure on $k$-sets $M'$ with $M' \cong M$ contains a structured sunflower $S \cong M$. We say that $M$ has the \emph{infinite sunflower property} if it has the infinite $k$-sunflower property for all $k \in \N_+$.
\end{defn}

(In \cite{AKM25}, the infinite $k$-sunflower property is referred to as being \emph{$k$-sunflowerable}.)

\begin{defn}[{\cite[Def.\ 49]{AKM25}}]
    Let $k \in \N_+$. Let $\mc{K}$ be a hereditary class of finite structures. We say that $\mc{K}$ has the \emph{finite $k$-sunflower property} if for each $B \in \mc{K}$ there is $C \in \mc{K}$ such that, for any structure on $k$-sets $C'$ with $C' \cong C$, we have that $C'$ contains a structured sunflower $B' \cong B$. We call such $C$ a \emph{witness} for $B$ of the finite $k$-sunflower property. We say that $\mc{K}$ has the \emph{finite sunflower property} if $\mc{K}$ has the finite $k$-sunflower property for all $k \in \N_+$.
\end{defn}

By a standard Ramsey-theoretic compactness argument, if a countable relational structure $M$ has the infinite $k$-sunflower property, then its age (the class of finite structures embeddable in $M$) has the finite $k$-sunflower property: see \cite[Thm.\ 53]{AKM25}.

Rephrased in our terminology, the classical results of Erd\H{o}s and Rado state that the countably infinite pure set has the infinite sunflower property, and the class of finite sets has the finite sunflower property.

We primarily discuss countable first-order structures (often in a relational language), and we will mostly be concerned with \Fr structures and classes. For us, a \emph{\Fr structure} is a countably infinite ultrahomogeneous structure, and we refer to the class of finite structures embeddable in a \Fr structure as a \emph{\Fr class}. (Recall that a countable structure $M$ is \emph{ultrahomogeneous} if each isomorphism between finitely generated substructures of $M$ extends to an automorphism of $M$.) We assume the reader is familiar with the basic theory here -- see \cite{Mac11}, \cite{Hod93} for more background. 

\subsection*{Canonical Ramsey properties}

Recall the \emph{canonical infinite Ramsey theorem} of Erd\H{o}s and Rado (below, we write $X^{(r)}$ for the set of $r$-subsets of $X$):
\begin{fact*}[{\cite[Thm.\ 1]{ER50}}]
    Let $r \in \N_+$ and $\chi : \omega^{(r)} \to \omega$. Then there are $I \sub r$ and infinite $M \sub \omega$ such that for all $A, A' \in M^{(r)}$, enumerated as $a^{}_0 < \cdots < a^{}_{r-1}$, $a'_0 < \cdots < a'_{r-1}$ using the natural ordering $<$ on $\omega$, we have $\chi(A) = \chi(A')$ if and only if $a^{}_i = a'_i$ for all $i \in I$. 
\end{fact*}
This theorem generalises the infinite Ramsey theorem to the case of $\omega$-many colours: one may not be able to find a monochromatic infinite set, but one is guaranteed an infinite set on which the colouring behaves in a canonical way. Note that, as with the infinite Ramsey theorem, a corresponding canonical finite Ramsey theorem follows via a compactness argument (see \cite[Section 1.5]{GRS91}).

As with structural versions of Ramsey's theorem (for example, the \NR theorem of \cite{NR77} and its generalisation due to Hubi\v{c}ka and \Jarik in \cite{HN19}), structural analogues of the above Erd\H{o}s-Rado canonical Ramsey theorem have been studied before: see for example \cite{NR78}, \cite{ENR84}, \cite{PV82}, \cite{PV83}, \cite{PV85}, \cite{Voi84}, \cite{Voi85}, \cite{DMT17}, \cite{Mas19}. In this paper, we shall only be concerned with canonical Ramsey properties for colourings of points:

\begin{defn} \label{d: heterochromatic}
    Let $M$ be a first-order structure. Given a colouring of the elements of $M$, we say that $N \sub M$ is \emph{heterochromatic} if each element of $N$ has a different colour.
\end{defn}
\begin{defn} \label{d: cpRP}
    Let $M$ be a \Fr structure. We say that $M$ has the \emph{canonical infinite point-Ramsey property} if:
    \begin{enumerate}
        \item[] for each colouring $\chi : M \to \omega$, there is a monochromatic or a heterochromatic copy of $M$.
    \end{enumerate}
    Let $\mc{K}$ be a \Fr class. We say $\mc{K}$ has the \emph{canonical point-Ramsey property} if:
    \begin{enumerate}
        \item[] for all $B \in \mc{K}$, there exists $C \in \mc{K}$ such that, for each colouring $\chi : C \to \omega$, there is a monochromatic or a heterochromatic copy of $B$ in $C$.
    \end{enumerate}
    (One could also define these properties more generally for countable structures and hereditary classes.)
\end{defn}

\subsection*{Main results}

The two main theorems of this paper are the following:
\begin{manualthm}{A}
    Let $M$ be a relational \Fr structure with strong amalgamation. Then the following are equivalent:
    \begin{enumerate}[label=(\roman*)]
        \item $M$ has the infinite sunflower property;
        \item $M$ has the infinite $2$-sunflower property;
        \item $M$ has the canonical infinite point-Ramsey property.
    \end{enumerate}
\end{manualthm}
\begin{manualthm}{B}
    Let $\mc{K}$ be a \Fr class. Then the following hold:
    \begin{enumerate}[label=(\roman*)]
        \item if $\mc{K}$ has the finite $2$-sunflower property, then it has the canonical point-Ramsey property;
        \item if $\mc{K}$ has the \emph{very} canonical point-Ramsey property, then it has the finite sunflower property.
    \end{enumerate}
\end{manualthm}

See also Proposition \ref{p: equiv cipRP}, where we give two other characterisations of the canonical infinite point-Ramsey property for relational \Fr structures with strong amalgamation, one of which is a strengthening of the well-known property of indivisibility.

We briefly explain some of the terminology used in the statements of the above two theorems. Recall that a \Fr class $\mc{K}$ has \emph{strong amalgamation} if, for all pairs of embeddings $B_0 \la A \ra B_1$ in $\mc{K}$, there exists an amalgam $B_0 \ra C \la B_1$ in $\mc{K}$ such that the images of $B_0$, $B_1$ intersect exactly in the image of $A$. (Some authors refer to this as \emph{disjoint amalgamation}.) We also say that a \Fr limit has strong amalgamation if its age does.

The \emph{very canonical point-Ramsey property} mentioned in Theorem B is defined in Definition \ref{d: vcpRP}. The name is consciously ad hoc: it is a strengthening of the canonical point-Ramsey property, specifically concocted to imply the finite sunflower property and to hold for a variety of examples.

We conjecture that an analogous finite version of Theorem A should also be true, improving on Theorem B: namely, we conjecture that the finite sunflower property is in fact equivalent to the canonical point-Ramsey property, with no ad hoc strengthening required. Unfortunately, we were not able to show this. See Section \ref{s: further qns}.

\subsection*{Structure of the paper}

We prove Theorem A in Section \ref{s: inf SP}. First, we introduce two properties, the \emph{galah property} (Definition \ref{d: galah}) and \emph{local replicability} (Definition \ref{d: top loc replic}) which are equivalent to the canonical infinite point-Ramsey property in the case of relational \Fr structures with strong amalgamation (Proposition \ref{p: equiv cipRP}). We then use these equivalences in the proof of Theorem A, and subsequently give a variety of examples and further observations. In Section \ref{s: finite SP}, we prove Theorem B, and we show that \Fr classes with free amalgamation and a single vertex isomorphism-type have the finite sunflower property (Proposition \ref{p: vertex-tr free amalg have finite SP}), as do many classes of finite metric spaces (Proposition \ref{p: vcpRP for Sms}).

\subsection*{The relationship to previous results.} 

The notion of structured sunflowers was introduced in \cite{AKM25} (also considered in the algebraic setting in \cite{AM25}), where some sufficient conditions were given for a structure to have the infinite sunflower property: in the terminology of Subsection \ref{ss: galah locally repl}, Ackerman, Karker and Mirabi showed in \cite{AKM25} that an indivisible locally replicable ultrahomogeneous structure has the infinite sunflower property, assuming the size of the structure to be a regular cardinal. Applying this, they showed for regular cardinals $\kappa$ that $\kappa$-saturated models of the theory of the dense linear order without endpoints have the infinite sunflower property, as do $\kappa$-saturated models of the theory of the Rado graph. They also showed that \Fr structures with the $3$-DAP (see Definition \ref{d: 3-DAP}) have the infinite sunflower property, and determined which $\kappa$-scattered linear orders and which countable linear orders have the infinite sunflower property.

\section{The infinite sunflower property for countably infinite structures} \label{s: inf SP}

\subsection{The galah property}

Recall the following well-known partition properties (originating in \cite{Fr00}, \cite{Cam90}):

\begin{defn} \label{d: indiv pigeonhole}
    Let $M$ be a relational structure. We say that:
    \begin{enumerate}[label=(\alph*)]
        \item \label{i: pigeonhole prop} $M$ has the \emph{pigeonhole property} if, for each partition $(C, D)$ of $M$, we have $C \cong M$ or $D \cong M$;
        \item \label{i: indiv} $M$ is \emph{indivisible} if, for each partition $(C, D)$ of $M$, we have that $C$ contains a copy of $M$ or $D$ contains a copy of $M$;
        \item \label{i: age-indiv} $M$ is \emph{age-indivisible} if, for each partition $(C, D)$ of $M$, we have $\Age(C) = \Age(M)$ or $\Age(D) = \Age(M)$.
    \end{enumerate}
    It is immediate that \ref{i: pigeonhole prop} $\implies$ \ref{i: indiv} and \ref{i: indiv} $\implies$ \ref{i: age-indiv}.
\end{defn}

We introduce a new property intermediate between the pigeonhole property and indivisibility, which to the best of our knowledge has not been studied before:

\begin{defn} \label{d: galah}
    Let $M$ be a relational structure. We say that $M$ has the \emph{galah property}\footnote{A \emph{galah} is a common Australian bird, one half of which resembles a pigeon. See Figure \ref{f:galah}.} if, for each partition $(C, D)$ of $M$, we have that $C \cong M$ or $D$ contains a copy of $M$.
\end{defn}

It is immediate that:
\begin{enumerate}[label=(\roman*)]
    \item \label{i: pigeonhole implies galah} any structure with the pigeonhole property has the galah property;
    \item \label{i: galah implies indiv} any structure with the galah property is indivisible.
\end{enumerate}

\begin{eg} \label{e: rg rt rorg}
    A short argument (see \cite[Prop.\ 3]{Cam97}) shows that the random graph has the pigeonhole property, and the same argument shows that the random $n$-hypergraph, random tournament and random oriented graph also have the pigeonhole property (see \cite{BD99}). Hence these structures all have the galah property.
\end{eg}
\begin{rem}
    Cameron (\cite[Prop.\ 4]{Cam97}) showed that the only countable graphs with the pigeonhole property are the Rado graph, $K_\omega$ and its complement. Bonato, Cameron and Deli\'{c} (\cite{BCD00}) showed that the only countable tournaments with the pigeonhole property are the random tournament, $\omega^\alpha$ and its reverse, where $\alpha$ is a countable non-zero ordinal. They also classified the preorders (quasi-orders) with the pigeonhole property: in particular, they showed that the partial orders (of any cardinality) with the pigeonhole property are exactly: infinite antichains, and $\omega^\alpha$ and its reverse for every non-zero ordinal $\alpha$.
\end{rem}
\begin{eg} \label{e: implications strict}
    We give the following examples showing that the reverse directions of the two implications \ref{i: pigeonhole implies galah}, \ref{i: galah implies indiv} above do not hold in general:
    \begin{enumerate}[label=(\alph*)]
        \item \label{ex: Q} $(\Q, <)$ has the galah property, but not the pigeonhole property;
        \item \label{ex: Hn} the generic $K_n$-free graph $\mathbb{H}_n$ ($n \geq 3$) is indivisible, but does not have the galah property;
        \item \label{ex: eq rln} Let $\mb{E}$ be a structure on a countable set given by an equivalence relation consisting of infinitely many infinite equivalence classes. Then $\mb{E}$ is indivisible, but does not have the galah property.
    \end{enumerate}
    \ref{ex: Q}: We first show the galah property. If $(C, D)$ is a partition of $\Q$ with $C \not\cong (\Q, <)$, then $C$ has an endpoint or there are $a, b \in C$ with $(a, b) \cap C = \varnothing$, so it is immediate that $D$ contains a copy of $(\Q, <)$. To see that $(\Q, <)$ does not have the pigeonhole property, take the partition $C = \{q \in \Q \mid q < 0\} \cup \{1\}$, $D = \Q \setminus C$.
    
    \ref{ex: Hn}: Indivisibility of $\mb{H}_n$ for $n \geq 3$ was first proved in \cite{ES89}, with the case $n = 3$ shown earlier in \cite{KR86}. To see that $\mb{H}_n$ does not have the galah property: take a vertex $v$ in $\mb{H}_n$, let $D$ be the set of vertices adjacent to $v$, and let $C = \mb{H}_n \setminus D$. Then $D$ does not contain any $K_{n-1}$, thus does not contain a copy of $\mb{H}_n$, and $C$ does not contain any vertex adjacent to $v \in C$, so $C \not\cong \mb{H}_n$.

    \ref{ex: eq rln}: Let $A_0, A_1, \cdots$ be the equivalence classes of $\mb{E}$. Take $v \in A_0$, and let $C = \{v\} \cup \bigcup_{i \geq 1} A_i$, $D = \mb{E} \setminus C = A_0 \setminus \{v\}$. Then $C \not\cong \mb{E}$ and $D$ does not contain a copy of $\mb{E}$, so $\mb{E}$ does not have the galah property. It is straightforward to check that $\mb{E}$ is indivisible.

    Note that example \ref{ex: eq rln} shows that the galah property is not preserved by lexicographic products, as $\mb{E}$ is the lexicographic product of an infinite set of inequivalent points with an infinite equivalence class, both of which have the galah property. Indivisibility is preserved under lexicographic products: see \cite[Proposition 2.21]{Mei16}.
\end{eg}

\begin{rem}
    Asymmetric versions of partition properties have been studied before. In \cite{ES91} the authors define that an infinite relational structure $M$ is \emph{weakly indivisible} if for each vertex partition $(C, D)$ of $M$, we have that $C$ contains a copy of $M$ or $\Age(D) = \Age(M)$. See also \cite{Sau14}.
\end{rem}

\begin{figure}
    \centering
    \includegraphics[width=0.3\linewidth]{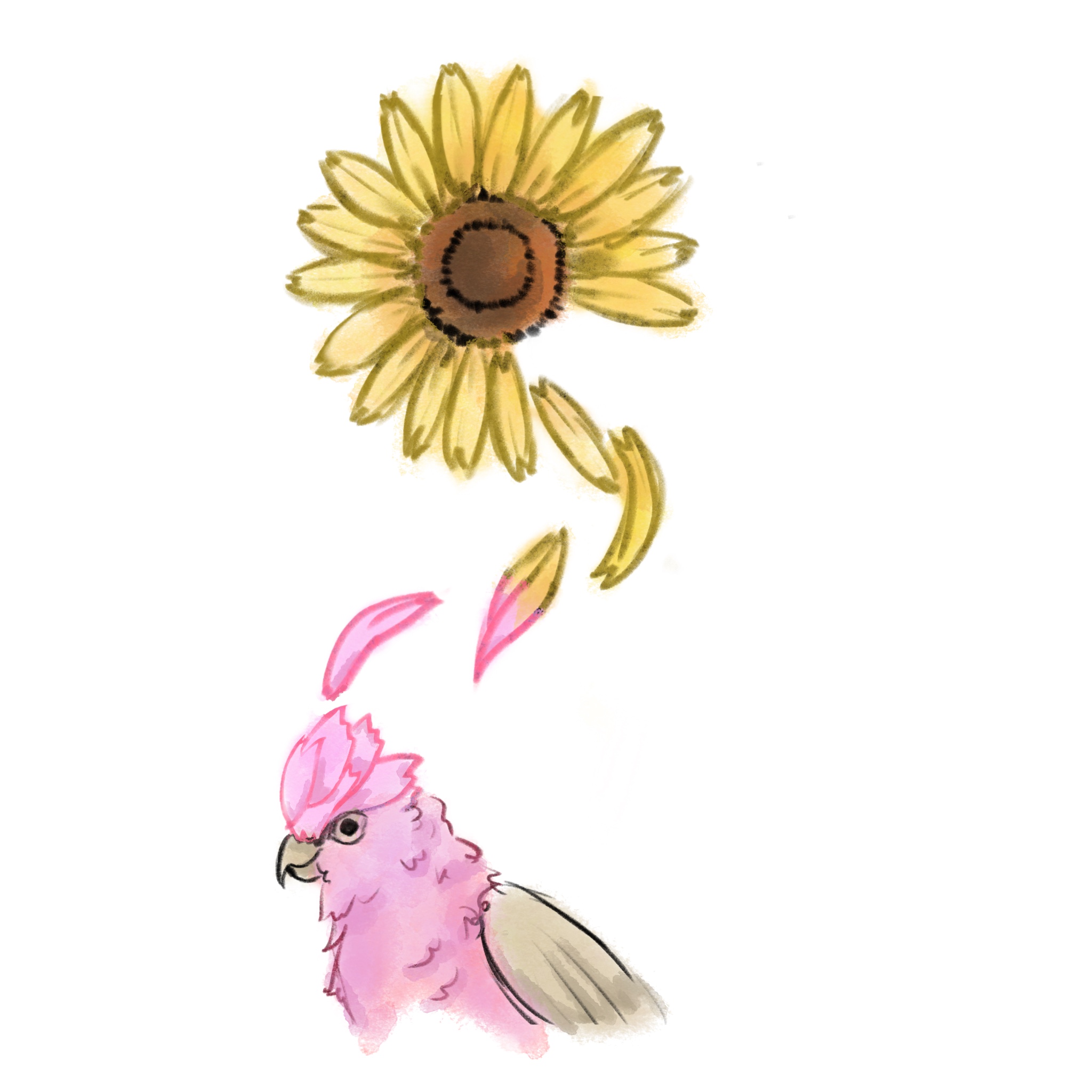}
    \caption{A galah and a sunflower. See Proposition \ref{p: equiv cipRP} and Theorem A.}
    \label{f:galah}
\end{figure}

\subsection{The galah property and local replicability} \label{ss: galah locally repl}

We first show some basic properties of structures with the galah property. 

\begin{lem} \label{l: galah implies strong amalg}
    Let $M$ be a relational structure with the galah property. Then:
    \begin{enumerate}[label=(\roman*)]
        \item \label{galah vt} all vertices of $M$ have the same quantifier-free type;
        \item \label{galah str amalg} if $M$ is a \Fr structure, then $M$ has strong amalgamation.
    \end{enumerate}
\end{lem}
\begin{proof}
    \ref{galah vt}: We show the contrapositive. Take $a \in M$, let $C$ be the set of vertices of $M$ with the same quantifier-free type as $a$, and let $D = M \setminus C$. Then $(C, D)$ violates the galah property.

    \ref{galah str amalg}: It suffices to show the strong extension property, namely: for each embedding $f : A \to B$ in $\Age(M)$ with $A \fin M$ and for each $V \fin M$ with $A \cap V = \varnothing$, there exists an embedding $f' : B \to M$ with $f' \circ f = \id_A$ and $\im(f') \cap V = \varnothing$. Consider the vertex partition $(M \setminus V, V)$ of $M$. By the galah property $M \setminus V \cong M$, so by the extension property of $M \setminus V$ there is an embedding $f' : B \to M \setminus V$ with $f' \circ f = \id_A$, and $f'$ is as required.
\end{proof}

\begin{rem}
    There exist indivisible \Fr structures without strong amalgamation: see \cite[Theorem 4(b)]{ES91}.
\end{rem}

\begin{defn} \label{d: top loc replic}
    Let $M$ be a countable relational structure. We define the following topology on $M$: for each finite substructure $A \fin M$ and $b \in M \setminus A$, we specify a basic open set $\{v \in M \mid \qftp(v / A) = \qftp(b / A)\}$, where $\qftp$ refers to the quantifier-free type. Note that this topology is metrisable: enumerate the finite substructures of $M$ as $A_0, A_1, \cdots$, and for distinct $v, w \in M$, define $d(v, w) = \frac{1}{i+1}$, where $i$ is least such that $\qftp(v/A_i) \neq \qftp(w/A_i)$. (Also define $d(v, w) = 0$ for $v = w$).

    If every non-empty open set in $M$ contains a copy of $M$, we say that $M$ is \emph{locally replicable}.
    
    (This was called \emph{universal duplication of quantifier-free types} in \cite[Def.\ 13]{AKM25}.)
\end{defn}

\begin{lem} \label{l: galah iff loc repl}
    Let $M$ be a relational \Fr structure. Then $M$ is locally replicable if and only if $M$ has the galah property.
\end{lem}
\begin{proof}
    $\Rightarrow$: Let $(C, D)$ be a partition of $M$. Suppose that $C \not\cong M$. So $C$ does not satisfy the extension property for $\Age(M)$ (including over the empty structure), and thus there is $A \fin C$ and a quantifier-free type $p$ over $A$ which is realised in $M$ but not in $C$. So $\{v \in M \mid \qftp(v / A) = p\} \sub D$, and as $\{v \in M \mid \qftp(v / A) = p\}$ is an open set, by assumption it contains a copy of $M$.

    $\Leftarrow$: We prove the contrapositive. Let $U = \{v \in M \mid \qftp(v / A) = p\}$ be a basic open set of $M$ not containing a copy of $M$, where $A \fin M$ and $p$ is a quantifier-free type. Let $V = M \setminus U$. As $A \sub V$ and there is no vertex of $V$ with type $p$, we have that $V$ does not have the extension property for $\Age(M)$. So $V \not\cong M$.
\end{proof}

\begin{rem}
    The $\Rightarrow$ direction of this lemma does not hold for first-order languages in general: the generic meet-semilattice is locally replicable but does not have the galah property. See Subsection \ref{ss: meet-semilattice} and the proof of Lemma \ref{l: S inf SP}.
\end{rem}

Note that Lemma \ref{l: galah iff loc repl} together with Lemma \ref{l: galah implies strong amalg} imply that any relational \Fr structure which is locally replicable must have strong amalgamation.

\begin{lem} \label{l: lr implies cipRP}
    Let $M$ be a relational \Fr structure. If $M$ is locally replicable, then $M$ has the canonical infinite point-Ramsey property.
\end{lem}
\begin{proof}
    Let $\chi : M \to \omega$. If there is no heterochromatic copy of $M$, then the extension property for $\Age(M)$ fails on heterochromatic sets. Thus there is heterochromatic $A \fin M$ and a basic open set $U = \{v \in M \mid \qftp(v/A) = p\}$ such that each $v \in U$ has the same colour as some point of $A$. As $M$ is locally replicable, there is $M' \sub U$, $M' \cong M$, and $M'$ has the galah property by Lemma \ref{l: galah iff loc repl} and is thus indivisible. The colouring $\chi|_{M'}$ only takes finitely many values, and so there is a monochromatic copy of $M$ inside $M'$. 
\end{proof}

\begin{notn}
    Given a quantifier-free type $p$ with parameter set $A$, enumerated as $\bar{a}$, and given an embedding $f$ of $A$, we write $f \cdot p$ for the type $\{\phi(x, f(\bar{a})) \mid \phi(x, \bar{a}) \in p\}$.
\end{notn}
\begin{defn}
    We say that a relational \Fr structure $M$ is \emph{transitive} if the natural permutation action $\Aut(M) \curvearrowright M$ is transitive; equivalently, $M$ is transitive if all vertices of $M$ have the same quantifier-free type.
\end{defn}

\begin{prop} \label{p: equiv cipRP}
    Let $M$ be a relational \Fr structure with strong amalgamation. Then the following are equivalent:
    \begin{enumerate}[label=(\roman*)]
        \item \label{i: cipRP} $M$ has the canonical infinite point-Ramsey property;
        \item \label{i: GP} $M$ has the galah property;
        \item \label{i: lr} $M$ is locally replicable.
    \end{enumerate}
\end{prop}
\begin{proof}
    \ref{i: GP} $\iff$ \ref{i: lr} is Lemma \ref{l: galah iff loc repl}, and \ref{i: lr} $\implies$ \ref{i: cipRP} is Lemma \ref{l: lr implies cipRP}.
    
    We now show \ref{i: cipRP} $\implies$ \ref{i: lr}. We prove the contrapositive. Suppose that $M$ is not locally replicable. First, consider the case where $M$ is not transitive. Define a colouring of $M$ by colouring each vertex with its isomorphism type. Then there is no monochromatic copy of $M$. As $M$ has strong amalgamation, each vertex isomorphism type must have infinitely many realisations in $M$, and so there is also no heterochromatic copy of $M$ and \ref{i: cipRP} does not hold. It remains to consider the case where $M$ is transitive.
    
    Let $U = \{v \in M \mid \qftp(v / A) = p\}$ be a basic open set of $M$ not containing a copy of $M$, where $A \fin M$ and $p$ is a quantifier-free type. Enumerate all embeddings $A \to M$ as $f_0, f_1, \cdots$. Define a colouring $\chi : M \to \omega$, $\chi(v) = \min\{i \in \omega \mid \qftp(v / f_i(A)) = f_i \cdot p\}$. (Note that there is an embedding $f : A \to M$ with $\qftp(v / f(A)) = p$ as $M$ is transitive.) As $M$ is ultrahomogeneous, each $f_i$ extends to an automorphism of $M$, and so as $U$ does not contain a copy of $M$, there is likewise no monochromatic copy of $M$. Suppose for a contradiction that there is a heterochromatic copy $N$ of $M$. As $N \cong M$ and $M$ has strong amalgamation, there is $j < \omega$ with $f_j(A) \sub N$ and infinitely many vertices $v \in N$ with $\qftp(v / f_j(A)) = f_j \cdot p$. But each of these vertices has a colour $\leq j$, contradicting that $N$ is heterochromatic.
\end{proof}

\subsection{Characterising the infinite sunflower property for \Fr structures}

\begin{lem} \label{l: inf 2SP implies cipRP}
    Let $M$ be a \Fr structure with the infinite $2$-sunflower property. Then $M$ has the canonical infinite point-Ramsey property.
\end{lem}
\begin{proof}
    Let $\chi : M \to \omega$ be a colouring. Note that we may assume for convenience that the domain of $M$ does not intersect the image of $\chi$. Define a structure $N$ with domain $\{\{v, \chi(v)\} \mid v \in M\}$, with relations given by specifying that the map $f : M \to N$, $f(v) = \{v, \chi(v)\}$ is an isomorphism. As $N$ has the infinite $2$-sunflower property by assumption, there is a structured sunflower $S \sub N$ with $S \cong M$. If the core of the sunflower $S$ is non-empty, then it consists of an integer $r$: thus $f^{-1}(S) \sub M$ is a copy of $M$ of colour $r$. If the core of $S$ is empty, then $f^{-1}(S)$ is heterochromatic.
\end{proof}

We now prove our main theorem characterising the infinite sunflower property:
\begin{manualthm}{A}
    Let $M$ be a relational \Fr structure with strong amalgamation. Then the following are equivalent:
    \begin{enumerate}[label=(\roman*)]
        \item \label{i: inf SP thm A} $M$ has the infinite sunflower property;
        \item \label{i: inf 2SP thm A} $M$ has the infinite $2$-sunflower property;
        \item \label{i: cipRP thm A} $M$ has the canonical infinite point-Ramsey property.
    \end{enumerate}
\end{manualthm}
Recall from Proposition \ref{p: equiv cipRP} that for $M$ a relational \Fr structure with strong amalgamation, we have:
\[
M \text{ has the canonical infinite point-Ramsey property}
\iff M \text{ has the galah property}
\iff M \text{ is locally replicable.}
\]
\begin{proof}[Proof of Theorem A]
    \ref{i: inf SP thm A} $\implies$ \ref{i: inf 2SP thm A} is trivial, and \ref{i: inf 2SP thm A} $\implies$ \ref{i: cipRP thm A} is Lemma \ref{l: inf 2SP implies cipRP}. We now show \ref{i: cipRP thm A} $\implies$ \ref{i: inf SP thm A}. We show by induction on $k$ that $M$ has the infinite $k$-sunflower property for all $k \in \N_+$. The case $k = 1$ is trivial. Let $k > 1$, and let $N$ be a structure on $k$-sets isomorphic to $M$. If $N$ contains a copy of $M$ whose domain consists of pairwise disjoint sets, then this forms a sunflower with empty core and we are done. Otherwise, there is $A \fin N$ with domain consisting of pairwise disjoint sets and a quantifier-free type $p$ over $A$ such that, for each $v \in N$ with $\qftp(v / A) = p$, we have $v \cap \bigcup A \neq \varnothing$. As $M$ is locally replicable by Proposition \ref{p: equiv cipRP}, there is a copy $N'$ of $M$ inside $\{v \in N \mid \qftp(v / A) = p\}$. Colour each $v \in N'$ with the colour $v \cap \bigcup A$. As $M$ is indivisible ($M$ has the galah property by Proposition \ref{p: equiv cipRP}, and the galah property implies indivisibility), there is $N'' \sub N'$, $N'' \cong M$, such that all elements of $N''$ have the same intersection with $\bigcup A$. Let $B$ denote this intersection, and let $W = \{v \setminus B \mid v \in N''\}$. Define a structure on the set $W$ by specifying that $N'' \to W$, $v \mapsto v \setminus B$ is an isomorphism: thus $W \cong M$. As $B \neq \varnothing$, all sets inside $W$ have size $k - |B|$. By the induction assumption $W$ contains a structured sunflower, and thus so does $N'' \sub N$.
\end{proof}

\begin{rem} \label{r: lr indiv implies inf SP}
    The same proof as in \ref{i: cipRP thm A} $\implies$ \ref{i: inf SP thm A} above shows that any \Fr structure which is both locally replicable and indivisible has the infinite sunflower property (with these assumptions we do not need to invoke Proposition \ref{p: equiv cipRP}, so strong amalgamation is not used, and we do not use the assumption of a relational language): \cite[Thm.\ 18]{AKM25} showed that any indivisible locally replicable ultrahomogeneous structure whose size is a regular cardinal has the infinite sunflower property.
\end{rem}

\begin{rem} \label{r: AKM}
    Ackerman informed us that he, Karker and Mirabi independently showed the following result (unpublished), by a generalisation of an argument for the Henson graph $\mb{H}_3$ sent to them by the authors of the present paper, combined with results of \cite{AKM25}: for a transitive relational structure $M$ of any infinite cardinality $\kappa$ with $\kappa^{< \kappa} = \kappa$ satisfying the condition that, for each $A \sub M$, $|A| < \kappa$, the orbit of any $v \in M \setminus A$ under the pointwise stabiliser of $A$ is of size $\kappa$, it is the case that the following equivalences hold: $M$ has the infinite sunflower property $\iff$ $M$ has the infinite $2$-sunflower property $\iff$ $M$ is indivisible and locally replicable.
\end{rem}

\subsection{Examples}

\begin{eg} \label{e: inf SP by galah}
    We apply Proposition \ref{p: equiv cipRP} and Theorem A to give examples of structures with and without the infinite sunflower property.
    \begin{itemize}
        \item From Example \ref{e: rg rt rorg}: the random graph, random tournament and random oriented graph all have the pigeonhole property and hence the galah property, and so all three structures have the infinite sunflower property.
        \item From Example \ref{e: implications strict}: the dense linear order $(\Q, <)$ has the galah property, and so has the infinite sunflower property.
        \item From Example \ref{e: implications strict}, the generic $K_n$-free graph $\mb{H}_n$ does not have the galah property, so does not have the infinite sunflower property. Likewise the free amalgam of $\omega$-many copies of $K_\omega$ does not have the infinite sunflower property.
    \end{itemize}
\end{eg}
\begin{eg} \label{e: inf SP by loc repl}
    We give two examples where we show the infinite sunflower property via local replicability, using Proposition \ref{p: equiv cipRP} and Theorem A.
    \begin{itemize}
        \item Let $M$ be the generic poset (the \Fr limit of the class of finite posets). We call a triple $(C, D, E)$ of finite subsets of $M$ \emph{valid} if:
        \begin{center}
            ($c < e$ for all $c \in C$, $e \in E$) $\,\wedge\,$
            ($\neg c > d$ for all $c \in C$, $d \in D$) $\,\wedge\,$
            ($\neg e < d$ for all $e \in E$, $d \in D$).
        \end{center}
        It is straightforward to check that the map $\xi$ from the set of external quantifier-free types over finite subsets of $M$ to the set of valid triples given by
        \[
        \xi(\qftp(b / A)) := (\{a \in A \mid a < b\}, \{a \in A \mid a \ic b\}, \{a \in A \mid a > b\})
        \]
        is a bijection. Let $U = \{v \in M \mid \qftp(v/A) = p\}$ be a basic open set of $M$, and let $(C, D, E) = \xi(p)$. To show that $U$ contains a copy of $M$, it suffices to show that for all $B \fin U$ and $v \in M \setminus B$, there is $u \in U$ with $\qftp(u/B) = \qftp(v/B)$. Let $B \fin U$, $v \in M \setminus B$, and let $(C', D', E') = \xi(\qftp(v / B))$. It is straightforward to check that $(C \cup C', D \cup D', E \cup E')$ is a valid triple, and so there is $u \in M$ with $\xi(\qftp(u / A \cup B)) = (C \cup C', D \cup D', E \cup E')$. We then have $u \in U$ and $\qftp(u / B) = \qftp(v / B)$, so $M$ is locally replicable, hence has the infinite sunflower property.
        \item Let $M$ be the generic ordered graph: namely, the \Fr limit of the class of finite linearly ordered graphs. It is straightforward to check that $M$ is locally replicable and hence has the infinite sunflower property.
    \end{itemize}
\end{eg}

\subsection{An example without strong amalgamation} \label{ss: meet-semilattice}
          
In this subsection, we give an example of an $\omega$-categorical relational \Fr structure without strong amalgamation which has the infinite sunflower property. By Lemma \ref{l: galah implies strong amalg} and Lemma \ref{l: galah iff loc repl}, as this structure does not have strong amalgamation, it does not have the galah property, and nor is it locally replicable.

\begin{defn} \label{d: canon str}
    Let $M$ be a \Fr structure with domain $D$. Let $\mc{L}_{\Can(M)}$ be a language given by taking an $n$-ary relation symbol $R_p$ for each $n \in \N_+$ and each quantifier-free type $p = \qftp^M(\bar{a})$ of a tuple $\bar{a} \in M^n$. Let $\Can(M)$ be the $\mc{L}_{\Can(M)}$-structure with domain $D$ where $R_p^{\Can(M)}$ consists exactly of the $n$-tuples with quantifier-free type $p$. We call $\Can(M)$ the \emph{canonical structure} of $M$ (see \cite[Sec.\ 3.1]{Mac11}). It is straightforward to check that $\Aut(M) = \Aut(\Can(M))$, that $\Can(M)$ is ultrahomogeneous, and that if $M$ is $\omega$-categorical then so is $\Can(M)$ (via the Ryll-Nardzewski theorem).
\end{defn}

\begin{lem} \label{l: inf SP for canon str}
    Let $\mc{L}$ be a first order language (not necessarily relational), and let $M$ be a \Fr $\mc{L}$-structure with the infinite sunflower property. Then $\Can(M)$ has the infinite sunflower property.
\end{lem}
\begin{proof}
    Let $N$ be an $\mc{L}_{\Can(M)}$-structure on $k$-sets with an isomorphism $\tau : \Can(M) \to N$. Define an $\mc{L}$-structure $M'$ on $\dom(N)$ so that $\tau$ is an $\mc{L}$-structure isomorphism $M \to M'$. Then it is easy to check that $\Can(M') = N$. As $M$ has the infinite sunflower property, there is a sunflower $S \sub M'$, $S \cong M$. As $S$ is ultrahomogeneous, it too has a canonical structure $\Can(S)$, and as $S \cong M$ we have $\Can(S) \cong \Can(M)$. But $N|_{\dom(S)} = \Can(M')|_{\dom(S)} = \Can(S)$, so we are done.
\end{proof}

Let $\mc{K}$ be the class of finite meet-semilattices: that is, each element of $\mc{K}$ is a finite partial order where any two elements $a, b$ have a meet $a \wedge b$ (a greatest lower bound). We consider each element of $\mc{K}$ as a structure in the language $\mc{L} = \{<, \wedge\}$, where $<$ is a binary relation symbol and $\wedge$ is a binary function symbol. It is not difficult to verify that $\mc{K}$ has strong amalgamation (we leave this to the reader), and we let $\mb{S}$ denote the \Fr limit of $\mc{K}$. We call $\mb{S}$ the \emph{generic meet-semilattice}.

\begin{lem} \label{l: S point-accessible}
    The generic meet-semilattice $\mb{S}$ satisfies the following:
    \begin{enumerate}[label=(\roman*)]
        \item \label{i: S pt substr} for each $v \in \mb{S}$, we have that $\{v\}$ is the domain of a substructure of $\mb{S}$;
        \item \label{i: S exts pt-ac} for all $A \sub B \fin \mb{S}$ with $A$, $B$ substructures of $\mb{S}$, there is a chain $A = B_0 \sub \cdots \sub B_{n-1} = B$ of substructures of $\mb{S}$ with $|\dom(B_{i+1}) \setminus \dom(B_i)| = 1$ for all $i < n-1$.
    \end{enumerate}
\end{lem}
\begin{proof}
    \ref{i: S pt substr}: immediate. \ref{i: S exts pt-ac}: it suffices to check that for each finite substructure $C \fin \mb{S}$, one can enumerate $C$ as $c_0, \cdots, c_{r-1}$ such that for $i, j < r$ we have $c_i \wedge c_j = c_k$ for some $k \leq \min\{i, j\}$, and this follows easily by induction: removing a maximal element of $C$ gives a substructure, which we may enumerate by the induction assumption, and then we place the maximal element last in the enumeration.
\end{proof}

We may define local replicability exactly as in Definition \ref{d: top loc replic}: note that we were careful to define basic open sets in terms of quantifier-free types over finite \emph{substructures}.

\begin{lem} \label{l: S loc repl}
    The generic meet-semilattice $\mb{S}$ is locally replicable.
\end{lem}
We leave the proof of Lemma \ref{l: S loc repl} to the reader -- see the similar proof for the generic poset in Example \ref{e: inf SP by loc repl}.

\begin{lem} \label{l: S inf SP}
    The generic meet-semilattice $\mb{S}$ has the infinite sunflower property.
\end{lem}
\begin{proof}
    First note that $\mb{S}$ does not have the galah property: take $a, b \in \mb{S}$ with $a, b$ incomparable, and consider the partition $(\mb{S} \setminus \{a \wedge b\}, \{a \wedge b\}$. Nonetheless, we will show that $\mb{S}$ is indivisible, proceeding similarly to the proof of Lemma \ref{l: galah iff loc repl}. Let $(C, D)$ be a partition of $\mb{S}$. If $C$ does not contain a copy of $\mb{S}$, then there is a finite substructure $A$ of $S$ with $\dom(A) \sub C$ (potentially $A = \varnothing$) and an embedding $f : A \to B$ with $B \in \Age(S)$ such that there is no $B' \sub C$ with $\qftp(B'/A) = f^{-1} \cdot \qftp(B/f(A))$. Using Lemma \ref{l: S point-accessible}, we may assume $|B \setminus f(A)| = 1$, and then the rest of the argument is as before. We then follow the proof of \ref{i: cipRP thm A} $\implies$ \ref{i: inf SP thm A} in Theorem A to show that $\mb{S}$ has the infinite sunflower property, again using Lemma \ref{l: S point-accessible} to show that if the extension property fails, it fails for a one-point extension, and noting that we only use local replicability and indivisibility of $\mb{S}$ in this proof (see Remark \ref{r: lr indiv implies inf SP}).
\end{proof}

\begin{prop} \label{p: Can(S)}
    The structure $\Can(\mb{S})$ is an $\omega$-categorical relational \Fr structure with the infinite sunflower property but without strong amalgamation. 
\end{prop}
\begin{proof}
    It was observed in Definition \ref{d: canon str} that $\Can(\mb{S})$ is $\omega$-categorical and ultrahomogeneous. The infinite sunflower property follows from Lemma \ref{l: inf SP for canon str} and Lemma \ref{l: S inf SP}. Take $a, b \in \mb{S}$ with $a, b$ incomparable, and let $p = \qftp(a, b, a \wedge b)$. Then there are only finitely many $c \in \Can(\mb{S})$ with $\Can(\mb{S}) \models R_p(a, b, c)$, and so as $\{a, b\}$ is the domain of a substructure of $\Can(\mb{S})$, we have that $\Can(\mb{S})$ does not have strong amalgamation.
\end{proof}

\subsection{Further observations regarding locally replicable \Fr structures} \label{ss: further obs locally repl}
          
\begin{lem}
    Let $M$ be an $\omega$-categorical relational \Fr structure. If each non-empty open set $U$ in $M$ satisfies $\Age(U) = \Age(M)$, then $M$ is locally replicable.
\end{lem}
\begin{proof}
    Let $U = \{v \in M \mid \qftp(v/A) = p\}$ be a non-empty open set in $M$. Let $B_0 \sub B_1 \sub \cdots$ be an increasing chain of finite substructures of $M$ such that $\bigcup_{i < \omega} B_i = M$, $B_0 = \varnothing$ and $|B_{i+1} \setminus B_i| = 1$ for all $i < \omega$. Let $T$ be a rooted directed tree where the set of vertices of level $i$ is $\{\qftp(\bar{v}/A) \mid \bar{v} \in U^i, \bar{v} \cong B_i\}$, which is a finite set (by $\omega$-categoricity and the Ryll-Nardzewski theorem), and where there is an out-edge from $p$ in level $i$ to $p'$ in level $i + 1$ if $p \sub p'$. By K\H{o}nig's lemma, the tree $T$ contains an infinite branch, and so by repeatedly applying the extension property we have that $U$ contains a copy of $M$.
\end{proof}

We now investigate \Fr structures $M$ satisfying the property that each basic open set is \emph{isomorphic} to $M$ (this is called being \emph{quantifier-free definably self-similar} in \cite{GP23}, and is referred to in \cite{AKM25} as $M$ having \emph{strong universal duplication of quantifier-free types}):

\begin{lem}
    Let $M$ be a relational \Fr structure in a binary language. Suppose that $M$ is locally replicable. Then each basic open set of $M$ is isomorphic to $M$.
\end{lem}
\begin{proof}
    Let $U = \{v \in M \mid \qftp(v/A) = p\}$ be a basic open set. As $U$ contains a copy of $M$ we have $\Age(U) = \Age(M)$. Let $B, C \fin U$ and let $f : B \to C$ be an isomorphism. As the language of $M$ is binary we have that $f \cup \id_A$ is also an isomorphism, and as $M$ is ultrahomogeneous $f \cup \id_A$ extends to some $g \in \Aut(M)$. As $g(U) = U$ we have that $U$ is ultrahomogeneous, and so $U \cong M$. 
\end{proof}

We give an example showing that the above lemma does not hold in general for higher-arity languages:

\begin{eg}
    We write $R$ for the $3$-hypergraph relation. Let $F = \{a, b, c, d, e\}$ be a $3$-hypergraph with edges as follows: $\{a, b, c, d\}$ is a complete $3$-hypergraph, and $F$ additionally has the edges $abe, cde$. Let $\mc{K}$ be the class of finite $3$-hypergraphs without $F$ as a (possibly non-induced) subgraph. As each pair of vertices of $F$ lie in some edge, the class $\mc{K}$ has free amalgamation. Let $M$ be the \Fr limit of $\mc{K}$. Let $a, b \in M$ be distinct, and take the basic open set $U_0 = \{v \in M \mid Rvab\}$. By the extension property of $M$, there exist $c, d \in U_0$ such that $\{a, b, c, d\}$ is complete, and so as $M$ omits $F$, for all $v \in U_0$ we have $\neg R cdv$. So $U_0 \not\cong M$.

    Now let $U = \{v \in M \mid \qftp(v/A) = p\}$ be a basic open set of $M$. We will show that $U$ contains a copy of $M$. We observe: for all $B \fin U$ such that no edge of $A \cup B$ contains exactly one vertex of $A$, and for all quantifier-free types $q = q(x / B)$ with a realisation in $M \setminus B$, there is a realisation $c \in U$ of $q$ such that no edge of $A \cup B \cup \{c\}$ contains exactly one vertex of $A$ (it is a straightforward check that one does not create a copy of $F$). Thus, writing $M$ as the union of an increasing chain of finite substructures and using the extension property, we find a copy of $M$ in $U$ as required.
\end{eg}

It was shown in \cite{GPS23} that, if $M$ is a transitive \Fr structure, then a well-known higher amalgamation property, the \emph{$3$-disjoint amalgamation property} (over $\varnothing$), implies that each basic open set of $M$ is isomorphic to $M$:

\begin{notn}
    As usual, for $n \in \N$ we assume $n = \{0, \cdots, n-1\}$. We write $n^{(k)}$ for the set of subsets of $n$ of size $k$.
\end{notn}

\begin{defn}[{\cite[Def.\ 3.1]{Kru19}}] \label{d: 3-DAP}
    Let $\mc{L}$ be a relational language. For each $I \subsetneq J \subsetneq 3$, let $f_{I, J} : A_I \to A_J$ be an embedding of finite $\mc{L}$-structures, and suppose that for distinct $I, I' \in 3^{(1)}$ we have $f_{I, I \cup I'} \circ f_{\varnothing, I} = f_{I', I \cup I'} \circ f_{\varnothing, I'}$ and $\im(f_{I, I \cup I'}) \cap \im(f_{I', I \cup I'}) = \im(f_{I, I \cup I'} \circ f_{\varnothing, I})$. We call the family of embeddings $(f_{I, J})_{I \subsetneq J \subsetneq 3}$ a \emph{$3$-disjoint family}.
    
    Let $(f_{I, J})_{I \subsetneq J \subsetneq 3}$ be a $3$-disjoint family. For each $J \in 3^{(2)}$ let $g_J : A_J \to B$ be an embedding of finite $\mc{L}$-structures, and suppose that $g_J \circ f_{J \cap J', J} = g_{J'} \circ f_{J \cap J', J'}$ for all distinct $J, J' \in 3^{(2)}$. We call $(g_J)_{J \in 3^{(2)}}$ a \emph{$3$-disjoint amalgam} of $(f_{I, J})_{I \subsetneq J \subsetneq 3}$.

    Let $\mc{K}$ be a relational \Fr class. If each $3$-disjoint family of embeddings in $\mc{K}$ has a $3$-disjoint amalgam in $\mc{K}$, we say that $\mc{K}$ has the \emph{$3$-disjoint amalgamation property ($3$-DAP)}. If each $3$-disjoint family $(f_{I, J})$ in $\mc{K}$ with $\dom(f_{\varnothing, I}) = \varnothing$ for all $I \in 3^{(1)}$ has a $3$-disjoint amalgam in $\mc{K}$, we say that $\mc{K}$ has the \emph{$3$-DAP over $\varnothing$}.
\end{defn}

The following is \cite[Proposition 2.23]{GPS23} (the proof is relatively straightforward, and we omit it):
\begin{lem} \label{l: 3-DAP implies loc repl}
    Let $M$ be a transitive relational \Fr structure whose age has the $3$-DAP over $\varnothing$. Then each basic open set of $M$ is isomorphic to $M$.
\end{lem}
\begin{rem}
    In \cite{GPS23}, this result is stated with the assumption that $\Age(M)$ has the $3$-DAP, but in fact only the $3$-DAP over $\varnothing$ is used in the proof. The above result also appears in \cite[Lemma 16]{AKM25} with the additional assumption of $\omega$-saturation (and the erroneous omission of transitivity).
\end{rem}

Thus via Lemma \ref{l: galah iff loc repl}, Lemma \ref{l: galah implies strong amalg}, Proposition \ref{p: equiv cipRP} and Theorem A we have:
\begin{prop} \label{p: 3-DAP implies galah}
    Let $M$ be a transitive relational \Fr structure with the $3$-DAP over $\varnothing$. Then $M$ has the galah property and the infinite sunflower property.
\end{prop}
\begin{eg}
    Let $r \in \N$ with $r \geq 2$. For $n \in \N_+$, we let $K^{(r)}_n$ denote the complete $r$-uniform hypergraph on $n$ vertices. It is straightforward to see that for $n > r \geq 3$, the generic $K^{(r)}_n$-free hypergraph (the \Fr limit of the class of finite $K^{(r)}_n$-free hypergraphs) has the $3$-DAP, and hence by Proposition \ref{p: 3-DAP implies galah} has the infinite sunflower property.   
\end{eg}
              
\section{The finite sunflower property} \label{s: finite SP}

We first define a certain canonical Ramsey property, the \emph{very canonical point-Ramsey property}, and show that it implies the finite sunflower property in the case of \Fr classes (Theorem B). We then show that free amalgamation classes with a single vertex isomorphism-type have the very canonical point-Ramsey property (Proposition \ref{p: vertex-tr free amalg have finite SP}), as do many classes of metric spaces (Proposition \ref{p: vcpRP for Sms}), and so these classes have the finite sunflower property.

The key idea in showing the finite sunflower property via a canonical Ramsey property is as follows. Given $B$ and a structure on $k$-sets $C'$ in which we wish to find a copy of $B$ as a sunflower, one would like to define vertex-colourings of $C'$ that enable us to compare elements of each $k$-set vertex, so that we can use a canonical Ramsey property to control the intersections of these $k$-sets. For example, one can colour each $k$-set vertex in $C'$ with its least element (giving potentially $\omega$-many colours), and then a monochromatic set has an element in its common intersection, which one would then like to remove so as to proceed via induction, as in the classical proof of the Erd\H{o}s-Rado sunflower lemma. However, if one can only find a heterochromatic set, there is not enough control over the potential intersections of the $k$-set vertices. To fix this, one would like to define colourings which are able to compare, e.g.\ the first element of one $k$-set and the third element of another. To be able to define such colourings in a consistent way, we introduce a partition of the vertex set of $C'$, so that we may define colourings like ``if the vertex is in the first set of the partition, look at its first element, and if the vertex is in the seventh set of the partition, look at its third element". This gives us enough information to find sunflowers. See the proof of Theorem B.

\begin{defn}
    Let $n \in \N$ with $n \geq 1$. Let $X$ be a set with partition $X = X_0 \sqcup \cdots \sqcup X_{n-1}$. We call an $n$-set $U \sub X$ a \emph{transversal} if $|U \cap X_i| = 1$ for all $i < n$.
\end{defn}

\begin{defn} \label{d: vcpRP}
    Let $\mc{K}$ be a \Fr class. We say that $\mc{K}$ has the \emph{very canonical point-Ramsey property} if for all $s \in \N_+$, for all $B \in \mc{K}$, there exists $C \in \mc{K}$ and a partition $V_0 \sqcup \cdots \sqcup V_{|B| - 1}$ of the domain of $C$ such that for any $s$ vertex-colourings $(\chi_r : C \to \omega \mid r < s)$, at least one of the following holds:
    \begin{itemize}
        \item there exist $r < s$ and $B' \sub C$, $B' \cong B$, such that $\chi_r$ is monochromatic on $B'$;
        \item there exists transversal $B' \sub C$, $B' \cong B$, such that for all $r < s$, the colouring $\chi_r$ is heterochromatic on $B'$.
    \end{itemize}
    We call such $C$ a \emph{witness} of the very canonical point-Ramsey property for $(s, B)$.
\end{defn}

\begin{manualthm}{B}
    Let $\mc{K}$ be a \Fr class. Then the following hold:
    \begin{enumerate}[label=(\roman*)]
        \item \label{i: finite 2SP implies cpR} if $\mc{K}$ has the finite $2$-sunflower property, then it has the canonical point-Ramsey property;
        \item \label{i: vcpRP implies finite SP} if $\mc{K}$ has the \emph{very} canonical point-Ramsey property, then it has the finite sunflower property.
    \end{enumerate}
\end{manualthm}
\begin{proof}
    \ref{i: finite 2SP implies cpR}: the proof is entirely analogous to the proof of Lemma \ref{l: inf 2SP implies cipRP}: given $B \in \mc{K}$, one takes a witness $C$ for $B$ of the finite $2$-sunflower property, and the rest of the argument is the same (replacing $M$ with $C$).
    
    \ref{i: vcpRP implies finite SP}: We show that $\mc{K}$ has the finite $k$-sunflower property for all $k \in \N_+$ via induction on $k$. The case $k = 1$ is immediate. Suppose $k > 1$. Let $B \in \mc{K}$. By the induction assumption, there is a witness $D \in \mc{K}$ for $B$ of the finite $(k-1)$-sunflower property. Let $s = k^{|D|}$, and let $C \in \mc{K}$ with $\dom(C) = V_{0 \vphantom{|D| - 1}} \sqcup \cdots \sqcup V_{|D| - 1}$ be a witness for $D$ of the very canonical point-Ramsey property for $s$-many colourings. Let $C'$ be a structure on $k$-sets with $C' \cong C$, and let $V'_{0 \vphantom{|D| - 1}} \sqcup \cdots V'_{|D| - 1}$ be the partition of $\dom(C')$ induced by the isomorphism $C \to C'$. We may assume that the $k$-sets forming the domain of $C'$ are $k$-sets of natural numbers. Each $k$-set $v \in C'$ is linearly ordered by the standard order $<$ on $\N$, and thus may be enumerated as $v_0, \cdots, v_{k-1}$ with $v_0 < \cdots < v_{k-1}$.

    For each function $f : |D| \to k$, we define a colouring $\chi_f : C' \to \omega$ as follows: for each $i < |D|$ and for each $v \in V'_i$, we define $\chi_f(v) = v_{f(i)}$. By how we defined $C$, one of the following holds:
    \begin{enumerate}[label=(\roman*)]
        \item \label{some f with mono piece copy} there is some $f \in k^{|D|}$ and $D' \sub C'$, $D' \cong D$, such that $\chi_f$ is monochromatic on $D'$;
        \item \label{transversal copy hetero for all f} there is $D' \sub C'$, $D' \cong D$, with $|D' \cap V'_i| = 1$ for all $i < |D|$, such that for all $f \in k^{|D|}$, the colouring $\chi_f$ is heterochromatic on $D'$.
    \end{enumerate}
    We consider each case separately.
    
    \ref{some f with mono piece copy}: As $\chi_f$ is monochromatic on $D'$, there is $\lambda \in \N$ such that $\lambda \in v$ for all $v \in D'$. Let $D''$ be the structure on $(k-1)$-sets with domain $\{v \setminus \{\lambda\} \mid v \in D'\}$ defined by specifying that $v \mapsto v \setminus \{\lambda\}$ is an isomorphism $D' \to D''$. By the definition of $D$, we have that $D''$ contains a structured sunflower $B'' \cong B$, and so $D' \sub C'$ contains a structured sunflower $B' \cong B$.

    \ref{transversal copy hetero for all f}: Take distinct $u, v \in D'$. We have $u \in V'_i$, $v \in V'_j$ for some distinct $i, j$. Suppose there is $\lambda \in \N$ with $\lambda \in u \cap v$. Then there is $f \in k^{|D|}$ with $u_{f(i)} = v_{f(j)} = \lambda$, contradicting that $D'$ is heterochromatic in $\chi_f$. So the elements of $D'$ are pairwise disjoint $k$-sets, and as $D'$ contains a copy of $B$ we are done. \qedhere    
\end{proof}

\subsection{Transitive free amalgamation classes}
          
We now show that free amalgamtion classes with a single vertex isomorphism-type have the very canonical point-Ramsey property. The below Propsition \ref{p: very canonical hypergraph witness} is a generalisation of a canonical Ramsey theorem for hypergraphs due to \Jarik and R\"{o}dl (\cite[Theorem 2.1]{NR78}). We also adapt their proof strategy, which uses the probabilistic method and is inspired by the well-known proof of Erd\H{o}s and Hajnal that there exist finite hypergraphs of arbitrarily large girth and chromatic number (\cite{EH66}). In \cite{ENR84}, Erd\H{o}s, \Jarik and R\"{o}dl also reprove \cite[Theorem 2.1]{NR78} via an explicit construction (a form of the partite construction).

\begin{prop} \label{p: very canonical hypergraph witness}
    Let $n, s, g \in \N$ with $n, g \geq 2$ and $s \geq 1$. Then there exists a finite $n$-uniform hypergraph $H$ with girth $\geq g$ and with a partition $V_0 \sqcup \cdots \sqcup V_{n-1}$ of its vertex set, such that for any $s$ vertex-colourings $(\chi_r : H \to \omega \mid r < s)$, at least one of the following holds:
    \begin{itemize}
        \item there exists $r < s$ and $i < n$ such that $\chi_r$ is monochromatic on some edge contained in $V_i$;
        \item there exists a transversal edge $e$ such that $e$ is heterochromatic in each $\chi_r$, $r < s$.
    \end{itemize}
\end{prop}

Note that, in the above, we do not require that $H$ is $n$-partite.

\begin{rem}
    In fact, Proposition \ref{p: very canonical hypergraph witness} is stronger than required -- note that in the definition of the very canonical point-Ramsey property (Definition \ref{d: vcpRP}), we do not require the monochromatic copy of $B$ to lie in some set $V_i$ of the partition.
\end{rem}

Before proving Proposition \ref{p: very canonical hypergraph witness}, we first prove an auxiliary lemma.

\begin{lem} \label{l: suitable k-sets}
    Let $n \in \N$ with $n \geq 2$. Then for all $a_1 \in (0, 1)$, there exist $a_0 \in (0, 1)$ and $C \in \N$ such that for any finite set $V$ partitioned into $n$ equal-size parts $V_0, \cdots, V_{n-1}$, each of size $c \geq C$, and for any colouring $\chi : V \to \omega$, at least one of the following holds:
    \begin{enumerate}[label=(\roman*)]
        \item \label{a_0 mono} there is $i < n$ such that $V_i$ contains $> a_0 c^n$ monochromatic $n$-sets;
        \item \label{1 - a_1 transversal} there are $> (1 - a_1)c^n$ transversal heterochromatic $n$-sets.
    \end{enumerate}
\end{lem}
\begin{proof}
    Let $\eps \in \R_+$ with $(1 - n\eps)^n > 1 - a_1$. Then there is $a_0 \in (0, 1)$ such that $\binom{x\eps}{n} - a_0 x^n > 0$ for sufficiently large $x$ (take $a_0$ so that the coefficient of $x^n$ is positive); let $C \in \N$ be such that $\binom{c\eps}{n} - a_0 c^n > 0$ for all $c \geq C$.
    
    Let $V$ and $c \geq C$ be as in the statement of the lemma. Suppose there is $i < n$ and a colour $q$ such that $V_i$ contains $\geq c \eps$ elements of colour $q$. Then there are $\geq \binom{c \eps}{n}$ $q$-coloured $n$-sets in $V_i$, and as $\binom{c \eps}{n} > a_0 c^n$, condition \ref{a_0 mono} holds.

    Otherwise, for all $i < n$, there are $< c \eps$ elements of each colour in $V_i$. So, as one sees by picking elements part by part of distinct colours, the number of heterochromatic transversals is $> c(c - c\eps) \cdots (c - (n-1)c\eps) > (c - nc\eps)^n = c^n(1 - n\eps)^n > (1 - a_1)c^n$.
\end{proof}

\begin{proof}[Proof of Proposition \ref{p: very canonical hypergraph witness}]
    Let $\eps \in \R_+$ with $\eps < \frac{1}{g}$. Let $c$ be a sufficiently large integer (to be determined later), and let $V$ be a finite set partitioned into $n$ parts $V_0, \cdots, V_{n-1}$, each of size $c$. We let $H$ be a randomly chosen $n$-uniform hypergraph with vertex set $V$, where we select edges from the set of $n$-sets of $V$ independently with probability $p = c^{1 - n + \eps}$. We take $c$ large enough so that $p < \frac{1}{2}$, and thus $p < 1 - p$.

    Let $m \in \N$, $m \geq 2$. We first count the number of ``potential $m$-cycles" of $H$. A sequence $v_0 e_0 \cdots v_{m-1} e_{m-1}$ of vertices $v_i$ and $n$-sets $e_i$ is a \emph{potential $m$-cycle} if the vertices are distinct and $\{v_i, v_{i+1}\} \sub e_i$ for all $i$ (with addition $\Mod{m}$). By first taking an $m$-tuple of vertices, and then for each $v_i, v_{i+1}$ taking $n-2$ other vertices to form an $n$-set, we see that the number of potential $m$-cycles is $< (cn)^m \binom{cn}{n-2}^m < (cn)^{m(n-1)}$. So $\mb{E}(m\text{-cycles in } H) < (cn)^{m(n-1)}p^m = n^{m(n-1)} c^{m\eps}$, and so if $m \leq g$ we have $\mb{E}(m\text{-cycles in } H) = o(c)$.

    Now take $s$ vertex-colourings $(\chi_r : H \to \omega \mid r < s)$. We say that an $n$-set is \emph{suitable} if it is contained in some $V_i$ and monochromatic in some $\chi_r$, or if it is a transversal and heterochromatic in each $\chi_r$, $r < s$. Applying Lemma \ref{l: suitable k-sets} with $a_1 = \frac{1}{2s}$, we obtain $a_0$ and $C$. We take $c \geq C$. Let $a = \min\{a_0, \frac{1}{2}\}$. Then we have $> ac^n$ suitable $n$-sets. The probability that $< c$ of $ac^n$ suitable $n$-sets are edges is \[(1-p)^{ac^n}\left(1 + \binom{ac^n}{1}\frac{p}{1-p} + \cdots + \binom{ac^n}{c-1}\left(\frac{p}{1-p}\right)^{c-1}\right) < (1-p)^{ac^n}c\binom{ac^n}{c} < (1-p)^{ac^n} c^{cn + 1},\]
    and as $1 - p < \exp(-p)$ (using the Taylor series expansion of $\exp(x)$), we have that this probability is $< \exp(-pac^n) c^{cn + 1}$. We may assume that all colourings take values in the vertex set of $H$, so there are $(cn)^{cns}$ $s$-tuples of colourings. Thus, the probability that there is an $s$-tuple of colourings with $< c$ suitable edges is \[< \exp(-pac^n) c^{cn + 1} (cn)^{cns} = \exp(-ac^{1 + \eps} + (cns + cn + 1)\log(c) + cns\log(n)),\]
    and the latter expression $\to 0$ as $c \to +\infty$.
    
    So, for sufficiently large $c$ there is some $H$ such that all $s$-tuples of colourings have $\geq c$ suitable edges, and the number of cycles of length $< g$ is $< c$. Removing one edge from each cycle of length $< g$, we obtain the desired hypergraph.
\end{proof}

\begin{defn}
    Let $\mc{L}$ be a relational language. Let $A$ be an $\mc{L}$-structure. Recall that the \emph{Gaifman graph} of $A$ is defined as follows: we define a binary relation $\Gamma^A$ on $\dom(A)$ by specifying that $\Gamma^A(u, v)$ for distinct $u, v \in A$, if $u, v$ occur in a tuple of a relation of $A$.

    Let $B_0 \xleftarrow{f_0} A \xrightarrow{f_1} B_1$ be a pair of embeddings of finite $\mc{L}$-structures. Recall that a strong amalgam $B_0 \xrightarrow{g_0} C \xleftarrow{g_1} B_1$ of $(f_0, f_1)$ is a \emph{free amalgam} of $(f_0, f_1)$ if, for each relation symbol $R \in \mc{L}$, we have $R^C = R^{g_0(B_0)} \cup R^{g_1(B_1)}$: equivalently, the Gaifman graph of $C$ is the disjoint union of the Gaifman graphs of $g_0(B_0)$, $g_1(B_1)$ over the Gaifman graph of the image of $A$. We say that a \Fr class of relational structures is a \emph{free amalgamation class} if it is closed under free amalgams.

    We say that a finite $\mc{L}$-structure $A$ is \emph{irreducible} if its Gaifman graph is a clique: equivalently, the structure $A$ is not the free amalgam of any two of its proper substructures over their intersection. (We adopt this terminology from \cite[Definition 2.1]{HN19}.)
\end{defn}

The following Lemma \ref{l: free amalg iff forb irred} is folklore (see, for example, \cite[Lemma 3.2.7]{Sin17} for a proof):
\begin{lem} \label{l: free amalg iff forb irred}
    Let $\mc{K}$ be a class of finite structures over a relational language. Then $\mc{K}$ is a free amalgamation class if and only if there exists a set $\mc{F}$ of irreducible finite $\mc{L}$-structures such that $\mc{K}$ is exactly the class of finite structures in which no element of $\mc{F}$ embeds. (In this case, we say that $\mc{K}$ is \emph{$\mc{F}$-free}.)
\end{lem}

\begin{defn}
    We say that a \Fr class $\mc{K}$ is \emph{transitive} if each vertex of each structure in $\mc{K}$ has the same quantifier-free type. Note that $\mc{K}$ is transitive if and only if its \Fr limit $M$ is transitive -- hence the terminology (which is admittedly a slight abuse of language).
\end{defn}

The proof of the following Proposition \ref{p: vertex-tr free amalg have finite SP} is inspired by the proof of \cite[Theorem 3.2]{NR78} (in particular, the idea therein of ``pasting a structure" into the edges of a hypergraph of large girth, which one can also find in \cite{NR78b}).

\begin{prop} \label{p: vertex-tr free amalg have finite SP}
    Let $\mc{L}$ be a relational language. Let $\mc{K}$ be a transitive free amalgamation class of $\mc{L}$-structures. Then $\mc{K}$ has the very canonical point-Ramsey property, and hence by Theorem B has the finite sunflower property.
\end{prop}
\begin{proof}
    Let $g = 4$. Let $s \in \N_+$, $B \in \mc{K}$. Let $n = |B|$. We may assume $n \geq 2$ as the case $n = 1$ is trivial. Let $H$ be an $n$-uniform hypergraph given by Proposition \ref{p: very canonical hypergraph witness} (using the stated values of $s$, $n$ and $g$). Let $C$ be an $\mc{L}$-structure with domain $\dom(H)$ given by taking each edge of $H$ and replacing it with a copy of the $\mc{L}$-structure $B$ (and not adding any other relations). As the girth of $H$ is $> 2$, each pair of distinct edges of $H$ intersects in at most one vertex, and each vertex of $B$ has the same quantifier-free type, so $C$ is well-defined. We refer to the copies of $B$ in $C$ resulting from $H$-hyperedges as \emph{$H$-copies} of $B$.
    
    Let $\mc{F}$ be a set of irreducible finite $\mc{L}$-structures such that $\mc{K}$ is the class of $\mc{F}$-free $\mc{L}$-structures (the existence of $\mc{F}$ is guaranteed by Lemma \ref{l: free amalg iff forb irred}). First observe that for $u, v \in C$, if $(u, v)$ is an edge of the Gaifman graph of $C$, then there is an $H$-copy of $B$ containing $u, v$. As the girth of $H$ is $> 3$, for $u, v, w \in C$, if $u, v, w$ form a clique in the Gaifman graph of $C$, then $u, v, w$ lie in the same $H$-copy of $B$. So $C$ is $\mc{F}$-free, and thus $C \in \mc{K}$. As $H$ satisfies the properties specified in Proposition \ref{p: very canonical hypergraph witness}, it is immediate that $C$ is a witness of the very canonical point-Ramsey property for $(s, B)$.
\end{proof}

\subsection{\texorpdfstring{$S$}{S}-metric spaces}

In this subsection, we show that certain \Fr classes of finite metric spaces have the finite sunflower property.

\begin{defn}
    Let $S \sub \R_+$ be non-empty. We say that a metric space $M$ is an \emph{$S$-metric space} if all non-zero distances lie in $S$: that is, for all distinct $a, b \in M$ we have $d(a, b) \in S$. We let $\mc{U}_S$ denote the class of all finite $S$-metric spaces.
\end{defn}

We give an extremely abbreviated treatment of basic facts around classes of finite $S$-metric spaces, based on \cite[Sect.\ 4.2.2]{HN19} (the results we present are originally from \cite{DLPS07}, \cite{Sau13}).

We first characterise which distance sets $S$ give \Fr classes of $S$-metric spaces:

\begin{defn}[{\cite{DLPS07}}]
    A non-empty set $S \sub \R_+$ satisfies the \emph{four-values condition} if, for all $b, b', c, c' \in S$ where there is $a \in S$ such that the triangles $B, C$ with distances $a, b, b'$ and $a, c, c'$ are $S$-metric spaces, there is a strong amalgam $D$ of $B, C$ over the edge of distance $a$ such that $D$ is an $S$-metric space. (Equivalently, there is $d \in S$ such that the triangles with distances $b, c, d$ and $b', c', d$ are metric spaces.)
\end{defn}

\begin{fact}[{\cite[Prop.\ 1.4]{DLPS07}}] \label{f: fvc equiv}
    Let $S \sub \R_+$ be non-empty and countable. The following are equivalent:
    \begin{itemize}
        \item $S$ satisfies the four-values condition;
        \item $\mc{U}_S$ has the amalgamation property;
        \item $\mc{U}_S$ has the strong amalgamation property.
    \end{itemize}
\end{fact}

For certain $S$, one can phrase the above result in terms of an algebraic condition:

\begin{defn}[{\cite{Sau13b}}]
    Let $S \sub \R_+$. For $a, b \in S$, we define $a \oplus b = \sup\{s \in S \mid s \leq a + b\}$.
\end{defn}
\begin{fact}[{\cite[Thm.\ 4.18]{HN19}}] \label{f: assoc iff fvc}
    Let $S \sub \R_+$ be non-empty and countable, and suppose that $a \oplus b \in S$ for all $a, b \in S$. Then $\oplus$ is associative if and only if $S$ satisfies the four-values condition (which in turn is equivalent to $\mc{U}_S$ having the strong amalgamation property by Fact \ref{f: fvc equiv}).
\end{fact}

\begin{defn}[{\cite[Def.\ 4.14]{HN19}}]
    Let $\mc{L}_S = \{d_s \mid s \in S\}$ be a relational language with each $d_s$ binary. We call an $\mc{L}_S$-structure $A$ an \emph{$S$-graph} if the $d_s^A$, $s \in S$, are disjoint graph relations: that is, each $d_s^A$ is irreflexive and symmetric and, for distinct $u, v \in A$, we have $d_s^A(u, v)$ for at most one $s \in S$. We will write $d(u, v) = s$ to mean that $d_s^A(u, v)$.

    Let $A$ be an $S$-graph. We say that an $S$-graph $B$ on the same domain as $A$ is a \emph{weak $S$-subgraph} of $A$ if $d_s^B(u, v) \implies d_s^A(u, v)$ for all $u, v \in B$, $s \in S$ (we use the term ``weak" to emphasise that $A$ need not be an induced subgraph of $B$). 
    
    We consider each $S$-metric space $M$ as an $S$-graph as follows: for distinct $u, v \in M$, we define $d_s^M(u, v)$ if $d(u, v) = s$.  We say that an $S$-graph $A$ is an \emph{$S$-metric graph} if it is a weak $S$-subgraph of some $S$-metric space $M$.
\end{defn}

\begin{fact}[{\cite[Prop.\ 4.19, rephrased]{HN19}}] \label{f: S-metric check cycles}
    Let $S \sub \R_+$ be non-empty and countable, and suppose that $a \oplus b \in S$ for all $a, b \in S$ and $\oplus$ is associative. Let $A$ be a finite $S$-graph. Then $A$ is an $S$-metric graph if and only if for each cycle $a_0 \cdots a_{n-1}$ of $A$, for all $i < n$ we have:
    \[d(a_i, a_{i+1}) \leq d(a_{i+1}, a_{i+2}) + \cdots + d(a_{i-1}, a_i), \tag{$\ast$}\] where index addition is taken $\Mod{n}$.
\end{fact}

We now use the above facts to show the finite sunflower property for certain classes of finite metric spaces.

\begin{prop} \label{p: vcpRP for Sms}
    Let $S \sub \R_+$ be non-empty and countable, and suppose that $a \oplus b \in S$ for all $a, b \in S$ and $\oplus$ is associative. Then $\mc{U}_S$ has the very canonical point-Ramsey property, and hence by Theorem B has the finite sunflower property.
\end{prop}
\begin{proof}
    Let $s \in \N_+$. Let $B \in \mc{U}_S$, and let $n = |B|$. We may assume $n \geq 2$, as the case $n = 1$ is trivial. Let $\Delta$ be the diameter of $B$ (the maximum distance between two points of $B$), and let $\mu$ be the minimum distance between two distinct points of $B$. Let $g \in \N$, $g > 2$, be such that $\Delta \leq (g - 1)\mu$. Let $H$ be an $n$-uniform hypergraph given by Proposition \ref{p: very canonical hypergraph witness} (using the stated values of $s$, $n$ and $g$), and let $C$ be an $\mc{L}_S$-structure with domain $\dom(H)$ given by taking each edge of $H$ and replacing it with a copy of the $\mc{L}_S$-structure $B$ (and not adding any other relations). As the girth of $H$ is $> 2$, each pair of distinct edges of $H$ intersects in at most one vertex, so $C$ is well-defined. Note that $C$ is an $S$-graph. We refer to the copies of $B$ in $C$ resulting from $H$-hyperedges as \emph{$H$-copies} of $B$.

    It now suffices to show that the $S$-graph $C$ is an $S$-metric graph. We use Fact \ref{f: S-metric check cycles}: it is straightforward to see that we need only check cycles $v_0 \cdots v_{n-1}$ where there are $H$-copies $B_0, \cdots, B_{n-1}$ of $B$ such that $v_i \in B_i \cap B_{i+1}$ for all $i < n$ (with index addition $\Mod{n}$). We have $\{v_i, v_{i+1}\} \sub B_{i+1}$ and thus $d(v_i, v_{i+1}) \leq \Delta$ for all $i < n$, and so condition ($\ast$) in Fact \ref{f: S-metric check cycles} holds by how we chose $g$. So $C$ is an $S$-metric graph, and thus is a weak $S$-subgraph of some $C' \in \mc{U}_S$, which witnesses the very canonical point-Ramsey property for $B$.
\end{proof}

\subsection{Examples of classes with the finite sunflower property}

\begin{eg} \label{e: classes w fSP}
    We now give some examples of \Fr classes with the finite sunflower property.
    \begin{itemize}
        \item As observed in the introduction, if a \Fr structure $M$ has the infinite sunflower property then its age has the finite sunflower property by a standard Ramsey-theoretic compactness argument (see \cite[Thm.\ 53]{AKM25}). So the classes of finite graphs, finite tournaments, finite oriented graphs, finite hypergraphs, finite posets, finite ordered graphs and finite meet-semilattices have the finite sunflower property (see Example \ref{e: inf SP by galah}, Example \ref{e: inf SP by loc repl} and Subsection \ref{ss: meet-semilattice}). Likewise, any transitive \Fr class with the $3$-DAP has the finite sunflower property (see Proposition \ref{p: 3-DAP implies galah}): for example, the class of finite $K_n^{(r)}$-free $r$-uniform hypergraphs with $n > r \geq 3$.
        \item Using Proposition \ref{p: vertex-tr free amalg have finite SP}, the classes of finite $K_n$-free graphs and finite Henson oriented graphs (where one forbids a certain set of tournaments) have the finite sunflower property.
        \item By Fact \ref{f: fvc equiv}, Fact \ref{f: assoc iff fvc} and Proposition \ref{p: vcpRP for Sms}, the following classes of finite metric spaces have the finite sunflower property:
        \begin{itemize}
            \item the class of finite metric spaces with rational distances (the \Fr limit is the rational Urysohn space);
            \item the class of finite metric spaces with distances in $\Q \cap [0, 1]$ (the \Fr limit is the rational Urysohn sphere);
            \item the class of finite metric spaces with integer distances;
            \item for each $n \in \N_+$, the class of finite metric spaces with distance set $\{0, \cdots, n\}$.
        \end{itemize} 
    \end{itemize}
    In particular, the class of $K_n$-free graphs gives an example of a class with the finite sunflower property whose \Fr limit does not have the infinite sunflower property.
\end{eg}

\begin{eg} \label{e: rb triangle-free}
    The classes in the first bulletpoint of Example \ref{e: classes w fSP} have indivisible \Fr limits. We now give an example (adapted from \cite[Ex.\ 11.2]{Sau20}) of a \Fr class with the finite sunflower property whose \Fr limit is divisible, and where the divisibility can be quickly proved. For this example one can easily show failure of the galah property, but we regard divisibility as being of independent interest. (Note that the rational Urysohn space and some generic Henson oriented graphs are also divisible: see \cite{Hjo08}, \cite{DLPS07}, \cite{DLPS08} and \cite{ES93}.)

    Let $\mc{L} = \{R, B\}$ be a language where $R, B$ are binary relation symbols, and let $\mc{K}$ be the class of finite $\mc{L}$-structures $A$ such that each of $R^A$, $B^A$ is a triangle-free graph relation and $R^A \cap B^A = \varnothing$. Informally, the class $\mc{K}$ consists of the finite  ``red/blue-edge-coloured graphs without monochromatic triangles" (note that a triangle containing a red edge and a blue edge is permitted). It is immediate that $\mc{K}$ is a transitive free amalgamation class, so it has the finite sunflower property by Proposition \ref{p: vertex-tr free amalg have finite SP}. Let $M$ be the \Fr limit of $\mc{K}$. It is straightforward to show that $M$ does not have the galah property, by an analogous argument to that of Example \ref{e: implications strict}\ref{ex: Hn} (take the red neighourhood of a vertex and the complement of this set), and so $M$ does not have the infinite sunflower property by Proposition \ref{p: equiv cipRP} and Theorem A.

    We now show that $M$ is divisible. One can show this using the general result \cite[Theorem 2.3]{Sau20}; we present a direct argument for the convenience of the reader, extracted from \cite[Sect.\ 10.1, Ex.\ 11.2]{Sau20}.
    
    For $u, v \in M$, we say that $uv$ is an \emph{edge} of $M$ if $R(u, v) \vee B(u, v)$. Let $v_0, \cdots$ be an enumeration of the domain of $M$.  Define:
    \begin{align*}
        C &= \{v_n \in M \mid (\ex i < n)(R(v_i, v_n) \wedge \all j < i, v_j v_i \text{ is not an edge})\},\\
        D &= \{v_n \in M \mid (\ex i < n)(B(v_i, v_n) \wedge \all j < i, v_j v_i \text{ is not an edge})\},\\
        E &= M \setminus (C \cup D).
    \end{align*}
    It is straightforward to see that $E$ contains no edges. Suppose for a contradiction that $C$ contains a copy of $M$. Then there exists $v_n \in C$ whose blue-neighbourhood in $C$ contains a copy $H$ of the generic red triangle-free graph (that is, each edge of $H$ is red and the red graph structure of $H$ is isomorphic to the generic triangle-free graph). Each $v \in H$ is red-adjacent to some $v_i$ with $i < n$; giving $v$ colour $i$, we have a vertex-colouring of $H$ in finitely many colours. But $H$ is age-indivisible: this is Folkman's theorem (\cite{Fol70}), or alternatively follows from the well-known \NR theorem which states that the generic order expansion of a free amalgamation class has the Ramsey property (see \cite{NR77}). So some colour set of $H$ contains a red edge, contradiction. An analogous argument swapping the colours shows that $D$ does not contain a copy of $M$, and so $M$ is divisible. 
\end{eg}
          
\subsection{Negative examples}
          
We will now give two examples of \Fr classes which do not have the finite sunflower property. We shall use the following two lemmas, the first of which is folklore (its proof is straightforward and left to the reader):

\begin{lem} \label{l: v-t age-indiv iff vR}
    Let $M$ be a transitive relational \Fr structure. Then $M$ is age-indivisible if and only if $\Age(M)$ has the \emph{point-Ramsey property}, that is:
    \begin{enumerate}
        \item[] for all $B \in \mc{K}$ there exists $C \in \mc{K}$ such that for each colouring $\chi : C \to 2$, there is a monochromatic copy of $B$.
    \end{enumerate}
\end{lem}

\begin{lem} \label{l: finite 2SP implies canon vR}
    Let $\mc{K}$ be a \Fr class with the finite $2$-sunflower property. Then $\mc{K}$ has the point-Ramsey property.
\end{lem}
\begin{proof}
    By part \ref{i: finite 2SP implies cpR} of Theorem B, the class $\mc{K}$ has the canonical point-Ramsey property. We now show that the canonical point-Ramsey property implies the point-Ramsey property. Given $B \in \mc{K}$, we may assume $|B| \geq 3$ (embedding $B$ inside a larger element of $\mc{K}$ if necessary). Let $C$ be a witness of the canonical point-Ramsey property for $B$, and let $\chi : C \to 2$. Then we may consider $\chi$ as a colouring $C \to \omega$, and we are guaranteed a monochromatic copy of $B$, as no two-colour heterochromatic copy of a structure of size $\geq 3$ can exist.
\end{proof}

\begin{eg}
    We now give three examples of \Fr classes which do not have the finite 2-sunflower property. The \Fr limit of \ref{i: ex eq rln} is indivisible, the \Fr limit of \ref{i: ex 2 eq rlns} is age-indivisible but not indivisible, and the \Fr limit of \ref{i: ex S2} is not age-indivisible. (The failure of the point-Ramsey property for \ref{i: ex S2} is folklore. The failure of indivisibility for \ref{i: ex 2 eq rlns} is essentially \cite[Example 2.15]{GPS23}.)

    \begin{enumerate}[label=(\roman*)]
        \item \label{i: ex eq rln} \emph{An equivalence relation consisting of infinitely many infinite equivalence classes}: let $\mb{E}$ be the structure defined in Example \ref{e: implications strict}\ref{ex: eq rln}. The colouring produced by giving each equivalence class of $\mb{E}$ a different colour shows that $\mb{E}$ does not have the canonical point-Ramsey property, and so $\Age(\mb{E})$ does not have the canonical point-Ramsey property (here we use a standard Ramsey-theoretic compactness argument). Thus $\Age(\mb{E})$ does not have the finite $2$-sunflower property by Lemma \ref{l: finite 2SP implies canon vR}. The indivisibility of $\mb{E}$ is an easy check.
        \item \label{i: ex 2 eq rlns} \emph{The free superposition of two equivalence relations, each consisting of infinitely many infinite equivalence classes}: let $\mc{L} = \{E_0, E_1\}$ be a relational language with $E_0$, $E_1$ binary, and let $\mc{K}$ be the class of finite $\mc{L}$-structures $A$ where $E_0^A$, $E_1^A$ are both equivalence relations. Then $\mc{K}$ is a strong amalgamation class: denote its \Fr limit by $M$. It is straightforward to see that $M$ is isomorphic to the structure $N$ on the Cartesian product $\N^3$ given by: for $u, v \in \N^3$, we specify $E_0^N(u, v)$ if $u_0 = v_0$ and $E_1^N(u, v)$ if $u_1 = v_1$. The partition of $N$ given by $\{v \in \N^3 \mid v_0 < v_1\}$ and its complement shows that $N$ is not indivisible. The age-indivisibility of $N$ follows straightforwardly from the product Ramsey theorem for sets (see \cite[Ch.\ 5, Thm. 5]{GRS91}), and hence by Lemma \ref{l: v-t age-indiv iff vR} the class $\mc{K}$ has the point-Ramsey property. But the colouring $\chi : N \to \omega$, $\chi(u) = u_0$ has no monochromatic or heterochromatic copy of the substructure of $N$ induced on $\{(0, 0, 0), (1, 0, 0), (1, 1, 0), (0, 1, 0)\}$. So, arguing as in example \ref{i: ex eq rln}, the class $\mc{K}$ does not have the finite $2$-sunflower property.
        \item \label{i: ex S2} \emph{The dense local order} $\mb{S}(2)$: this is a tournament whose domain consists of the points on the unit circle with rational argument, and where $u \ra v$ if the angle subtended at the origin by the anticlockwise arc from $u$ to $v$ is $< \pi$. The tournament $\mb{S}(2)$ is one of the three countable ultrahomogeneous tournaments, the other two being the random tournament and $(\Q, <)$ -- see \cite{Lac84}, \cite[Section 3.3]{Cam90}. To see that $\mb{S}(2)$ is not age-indivisible, observe that it contains a directed $3$-cycle as a substructure, and consider the partition of $\mb{S}(2)$ given by the out-neighbourhood of a vertex and the complement of this set: each of these is linearly ordered by the tournament relation, and so cannot contain a directed $3$-cycle. Thus by Lemma \ref{l: v-t age-indiv iff vR} and Lemma \ref{l: finite 2SP implies canon vR} the age of $\mb{S}(2)$ does not have the finite $2$-sunflower property.
    \end{enumerate}
\end{eg}
              
\section{Further questions} \label{s: further qns}

We regard the most pressing question as being whether one can characterise the finite sunflower property in a similar manner to the infinite sunflower property. Recall the definition of the canonical point-Ramsey property from Definition \ref{d: cpRP}, and recall that we showed in part \ref{i: finite 2SP implies cpR} of Theorem B that the finite $2$-sunflower property implies the canonical point-Ramsey property. We make the following conjecture:

\begin{conj} \label{cj: fin SP iff cvRP}
    Let $\mc{K}$ be a \Fr class. We conjecture:
    \begin{center}
        $(\ast)$ \quad $\mc{K}$ has the canonical point-Ramsey property $\iff$ $\mc{K}$ has the finite sunflower property.
    \end{center}
    
    We specifically conjecture that $\mc{K}$ has the canonical point-Ramsey property if and only if $\mc{K}$ has the very canonical point-Ramsey property, which would imply $(\ast)$ by Theorem B. We also conjecture that $\mc{K}$ has the finite $2$-sunflower property if and only if $\mc{K}$ has the finite sunflower property: note that part \ref{i: finite 2SP implies cpR} of Theorem B together with $(\ast)$ would imply this.
\end{conj}
              
It would also be interesting to know which indivisible \Fr structures have the galah property, out of those that have already been considered in the literature. For example:

\begin{qn}
    Which Henson oriented graphs have the galah property? Which \Fr structures with free amalgamation have the galah property? Which \Fr metric spaces have the galah property? (See \cite{ES93}, \cite{Sau20}, \cite{DLPS07}.)
\end{qn}
           
\section*{Acknowledgements}
              
The first author would like to thank Lionel Nguyen Van Th\'{e} for some enlightening conversations regarding canonical Ramsey properties, Jan Hubi\v{c}ka for telling him about Example \ref{e: rb triangle-free}, and Samuel Braunfeld for a number of helpful comments on a draft of this paper (in particular, the suggestion to consider $S$-metric spaces). We would like to thank Nathanael Ackerman and Mostafa Mirabi for several discussions during the course of this project (see Remark \ref{r: AKM}) -- together with Leah Karker, they are currently investigating infinite structured sunflowers in the uncountable context, without necessarily assuming ultrahomogeneity. Finally, we would like to thank Shujie Yang for her illustration of a galah (Figure \ref{f:galah}).

\bibliographystyle{alpha}
\bibliography{references}

\end{document}